\documentclass{article}
\usepackage{hyperref}
\usepackage{amsmath}
\usepackage{amssymb}
\usepackage{amsthm}
\usepackage{mathrsfs}
\usepackage{color}

\newcommand{\ZZ}{\mathbb{Z}}
\newcommand{\NN}{\mathbb{N}}
\newcommand{\PP}{\mathbb{P}}

\newcommand{\RR}{\mathbb{R}}
\newcommand{\T}{\mathbb{T}}

\newcommand{\F}{\mathbb{F}}

\newcommand{\K}{{\cal K}}
\newcommand{\C}{{\cal C}}
\newcommand{\hH}{{\cal H}}

\newcommand{\V}{{\cal V}}

\newcommand{\X}{{\cal X}}

\newcommand{\Y}{{\cal Y}}
\newcommand{\Z}{{\cal Z}}

\newcommand{\aS}{{\cal S}}

\newcommand{\Os}{{\Omega}}
\newcommand{\Int}{{\rm int}}
\newcommand{\Cl}{{\rm cl}}
\newcommand{\vp}{{\vec p}}
\newcommand{\vq}{{\vec q}} 
\newcommand{\wI}{{\widehat I}}
\newcommand{\1}{{\bf 1}}

\newtheorem{lemma}{Lemma}
\newtheorem{remark}{Remark}
\newtheorem{proposition}{Proposition}
\newtheorem{corollary}{Corollary}

\begin{document}
\begin{center}
{\Large Regenerative processes for Poisson zero polytopes}

\vspace{1cm}

{Servet Mart{\'i}nez}\\
{Departamento Ingenier{\'\i}a Matem\'atica and Centro
Modelamiento Matem\'atico,\\  Universidad de Chile,\\
UMI 2807 CNRS, Casilla 170-3, Correo 3, Santiago, Chile.\\
Email: smartine@dim.uchile.cl} 

\vspace{0.5cm}

{Werner Nagel}\\
{Friedrich-Schiller-Universit\"at Jena,\\
Institut f\"ur Mathematik,\\
Ernst-Abbe-Platz 2,
D-07743 Jena, Germany.\\
Email: werner.nagel@uni-jena.de} 

\end{center}

\begin{abstract}
Let $(M_t: t > 0)$ be a Markov process of tessellations of $\RR^\ell$ and $(\C_t:\, t > 0)$ the process of their zero cells (zero polytopes) which has the same distribution as the corresponding process for Poisson hyperplane tessellations. Let $a>1$. 
Here  we describe
the stationary zero cell process $(a^t \C_{a^t}:\, t\in \RR)$
in terms of some regenerative structure and we prove that 
it is a Bernoulli flow. An important application are the STIT tessellation processes.
\end{abstract}

\bigskip


\bigskip


\section{Introduction}
\label{Sec1}

Let us consider a process of Poisson hyperplane tessellations of the euclidean space $\RR^\ell$,
for some $\ell\ge 1$. It is generated by a spatio-temporal Poisson process of hyperplanes marked with birth times.
Let $\hH$ be the space of hyperplanes  in $\RR^\ell$, endowed with the Borel $\sigma-$field
associated to the Fell topology. 

For locally finite and translation invariant measure $\Lambda$ on $\hH$ consider the
Poisson process $\hat{X}$ on $\hH \times [0,\infty )$ 
with intensity measure $\Lambda \otimes \lambda_+$, where $\lambda_+$
denotes the Lebesgue measure on $[0,\infty )$. Then we define the process
$$(\hat{X}_t:\, t>0)$$
 with 
$$\hat{X}_t =\{ (h,s)\in \hat{X} :\, s \leq t \} $$
and the process
$$(X_t:\, t>0)$$ where $X_t$ is the Poisson hyperplane tessellation (PHT)
generated by $\{ h:\, (h,s)\in \hat{X} ,\, s \leq t \} .$

Let $a>1$. The renormalized tessellation valued processes 
$(a^t X_{a^t}: t\in \RR)$ is a time
stationary Markov process. The main object of our study is 
its zero cell process.

Another essential motivation to investigate to this zero cell process  comes from the STIT tessellations.
A STIT tessellation process $Y=(Y_t: t > 0)$ is a Markov
process taking values in the space of tessellations on $\RR^\ell$,
 and it was first defined
in \cite{nw}. There it was also shown that the zero cell processes of  PHT and  STIT (with the same measure $\Lambda$) are identically distributed. And also for STIT, the renormalized processes 
$\Z=(\Z_t:=a^t Y_{a^t}: t\in \RR)$ is time
stationary.

Denote by $\C_t$ the zero cell (or zero polytope or Crofton polytope) which is the polytope in $X_t$ or $Y_t$, respectively, 
containing the origin. The process $\C=(\C_t:\, t > 0)$ is 
well-defined a.e.
So, we will study the process
$\Gamma=(\Gamma_t:=a^t \C_{a^t}: t\in \RR)$ 
which is a factor of $\Z$.

\medskip

In the present paper, we mainly rely on results and methods developed for STIT tessellations. But note that, regarding the zero cell processes, many formulations can be easily translated for PHT replacing the operation $\boxplus$, iteration of tessellations, by the operation of superposition. E.g.
(\ref{iterate22a}) appears as

\begin{equation}
\label{iterate22aPHT}
X_{t+s} \sim X_t\sqcup X'_s \ \mbox{ for all }t,s>0\, ,
\end{equation}
where $X_t\sqcup X'_s$ is the PHT generated by $\hat{X}_t\cup \hat{X}'_s$ for independent Poisson hyperplane (marked with birth times) processes $\hat{X}_t$ and $\hat{X}'_s$.

\medskip

We note that the distribution of the zero cell $C^*$ 
of a random tessellation is determined by the function
$$
(\PP(C^*\supset K): K\subset \RR^\ell, K \hbox{ compact and convex} )
$$
(cf. \cite{mol}, Theorem 7.8, where a corresponding
proposition is shown for the so-called containment functional).
In the present paper, for fixed compact and convex $K$, we 
consider the $0-1$ stationary process
$({\bf 1}_{\{\Gamma_n\supset K\}}: n\in \ZZ)$ associated to the
zero cell of $\Z$. We prove
that it is a regenerative process in the state $1$ and we study
some of its properties. 
A main one is Proposition \ref{coditionalnew}, where we 
construct the stationary process
$({\bf 1}_{\{\Gamma_n\supset a \,K\}}: n\in \ZZ)$
starting from $({\bf 1}_{\{\Gamma_n\supset K\}}: n\in \ZZ)$.
So, by recurrence we can construct the family of regenerative 
processes
$({\bf 1}_{\{\Gamma_n\supset a^i \,K\}}: n\in \ZZ)$ with $i\ge 1$,
starting from $({\bf 1}_{\{\Gamma_n\supset K\}}: n\in \ZZ)$.

\medskip

On the other hand, in \cite{m} and \cite{mn1}
it is shown that $\Z$ is isomorphic to a
time-continuous Bernoulli flow with infinite entropy.
Being $\Gamma$ a factor of $\Z$, it is also Bernoulli.
For completeness of the description of the process $\Gamma$ 
we supply the main ideas leading to the Bernoulli property
that shares many phenomena with regeneration.

\subsection{ Notation and some basic facts}
\label{sbsb1}

Let us fix some notation: $\ZZ$ is the set of integers, 
$\ZZ_+=\{n\in \ZZ: n\ge 0 \}$ and $\NN=\{n\in \ZZ: n>0 \}$.
For a finite set $I$ we denote by $|I|$ the number of its elements.
For $C\subseteq \RR^\ell$ we denote by $\Int \, C$ its interior,
by $\Cl \, C$ its closure
and by $\partial  C={\Cl \, C }\setminus {\Int \, C}$ its boundary.
For random elements we use $\sim$
to mean 'identically distributed as', or 'distributed as'.

\medskip

A metric space $(\X,d)$ is Polish if it is complete and separable.
A countable product space $\Pi_{l\in \NN}\X_l$  of Polish   
spaces is itself Polish. If $\X$ is a topological space then
${\cal B}(\X)$ denotes the Borel $\sigma-$field.

\medskip

We will always consider complete probability spaces
$(\X,{\cal B},\nu)$, that is, ${\cal B}$ contains all $\nu-$negligible sets. 
A  space $(\X,{\cal B},\nu)$ is a Lebesgue probability space if  it is isomorphic to the unit 
interval $[0,1]$ endowed with ${\cal B}([0,1])$, and a probability 
measure which is a convex combination of the Lebesgue measure 
and a pure atomic measure.
If $(\X,d)$ is a Polish space and $\nu$ is a probability 
measure on $(\X,{\cal B}(\X))$, then  
$(\X,{\cal B}(\X),\nu)$ is Lebesgue, see \cite{delar}. So, if  
$\X'\in {\cal B}(\X)$ is a nonempty Borel set  
and $\nu'$ is a probability measure on $(\X',{\cal B}(\X'))$ 
then $(\X',{\cal B}(\X'),\nu')$ is a Lebesgue probability space. 

\medskip

Let $(\X,d)$ be a metric space and $D_{\X}(\RR_+)$ 
be the space of c\`adl\`ag 
(right continuous with left limits) trajectories 
with values in $\X$ and time $\RR_+ =[0,\infty )$. 
The space $D_{\X}(\RR_+)$ endowed with the Skorohod topology
is metrizable, see \cite{ek} where the usual metric
is given. Also in Theorem 5.6 of Ch. 3 ibidem,
it is proven that if $(\X, d)$ is separable or a Polish 
space, then the metric space $D_{\X}(\RR_+)$ is also separable
or a Polish space respectively, when endowed with the usual metric.
The Borel $\sigma-$field ${\cal B}(D_{\X})$ associated with
$D_{\X}(\RR_+)$ is generated by the class of cylinders.
We can replace the time set $\RR_+$ by $\RR$ in these 
definitions and properties.

\subsection{Elements of ergodic theory }
\label{suberg}

An abstract dynamical system (d.s.) $(\Os, {\cal B}(\Os),\mu, \psi)$ is 
such that $(\Os, {\cal B}(\Os),\mu)$ is a Lebesgue probability space and 
$\psi:\Os\to \Os$ preserves $\mu$, i.e. $\mu\circ \psi^{-1}=\mu$. We will 
denote it by $(\Os, \mu, \psi)$. Let $(\Os, \mu, \psi)$ and
$(\Os',\mu', \psi')$ be two d.s. The measurable map
$\varphi:\Os\to \Os'$ is a factor map if
$\varphi\circ \psi=\psi'\circ \varphi$ $\; \mu-$a.e.
and  $\mu\circ \varphi^{-1}=\mu'$. And if $\varphi$ is also bijective
a.e. then it is an isomorphism.

\medskip

Let $(\aS, {\cal B}(\aS))$ be a Polish space and $L=\ZZ_+$ or $L=\ZZ$.
The shift $\sigma_{\aS}:\aS^L\to \aS^L$,
$\sigma_{\aS}(x)_n=x_{n+1}$ for $n\in L$, is measurable and a d.s.
$(\aS^L, \mu, \sigma_{\aS})$ is a shift system.
A stationary sequence $\Y^d=(\Y_n: n\in L)$ with state space $\aS$
and distribution $\mu^{\Y^d}$ on $\aS^L$ is the shift system
$(\aS^L, \mu^{\Y^d}, \sigma_{\aS})$.
A factor map $\varphi:\aS^\ZZ\to {\aS'}^\ZZ$ is non-anticipating
if $\mu-$a.e. in $x\in \aS^\ZZ$ the coordinate $(\varphi (x))_n$ 
only depends on $(x_m: m\le n)$.
Let $\nu_\aS$ be a probability measure on $(\aS, {\cal B}(\aS))$, then
$\sigma_\aS$ preserves the product measure $\nu_\aS^{\otimes L}$, and
$(\aS^L,\nu_\aS^{\otimes L},\sigma_\aS)$ is a Bernoulli shift.
A d.s. is said to be Bernoulli if it is isomorphic to a Bernoulli shift.
The Ornstein isomorphism theorem states that two-sided
Bernoulli shifts with the same entropy are isomorphic 
(see \cite{orns10} and \cite{orns11}).

\medskip

A flow (or continuous time d.s.) $(\Os,\mu, (\psi^t))$ satisfies
$\mu\circ (\psi^t)^{-1}=\mu$ for all $t\in \RR$,
$\psi^{t+s}=\psi^{t}\circ \psi^{s}$ $\mu-$a.e. for $t,s\in \RR$ and the map
$[0,\infty)\times \Os\to \Os$ with $(t,\omega)\mapsto \psi^t(\omega)$ is 
measurable.
Its entropy of is the entropy of $(\Os,\mu, \psi^1)$.
The shift flows are defined by the shift transformations
$\sigma^t(x_s: s\in \RR)=(x_{s+t}: s\in \RR)$, $t\in \RR$.
A stationary random process $\Y=(\Y_t: t\in \RR)$
with c\`adl\`ag trajectories with marginals in a
Polish space, defines a shift flow.
A Bernoulli flow $(\Os, \mu, (\psi^t))$
is a flow such that $(\Os,\mu, \psi^1)$ is isomorphic
to a Bernoulli shift. The isomorphism theorem for Bernoulli flows,
see \cite{orns2}, states that two Bernoulli
flows with the same entropy are isomorphic.

\subsection{ The space of tessellations }
\label{Sub1.1}

A polytope is the compact convex hull of a finite point set. 
By definition, a {\it tessellation} $T$ of 
$\RR^\ell$ is a countable family of polytopes with nonempty interior
called the cells of $T$, 
we set $T=\{C(T)^l:l\in \NN \}$, which satisfies:
\begin{eqnarray*}
&{}& (i)\; \RR^\ell=\bigcup_{n\in \NN} C(T)^l \;\; \hbox{(covering)},\\
&{}& (ii) \; \Int C(T)^l\cap \Int C(T)^m=\emptyset \hbox{ if } l\neq m
\;\; \hbox{(disjoint interiors)},\\ 
&{}& (iii)\;  \big|\{l\in \NN: C(T)^l\cap K \neq \emptyset 
\}\big|<\infty, \,
\forall \hbox{ compact } K\subset \RR^\ell \;\;  
\hbox{(locally finite)}. 
\end{eqnarray*}
Let $\T$ be the space of tessellations of $\RR^\ell$.
The boundary of a tessellation is 
$\partial T=\cup_{l\in \NN}\, \partial {C(T)^l}$. Note that
$T$ is determined by $\partial T$.

\medskip

Let $b\neq 0$. For $A\subset \RR^\ell$ we set
$bA=\{bx: x\in A\}$. Then for $T \in \T$ and $B\subset \T$
we define $b\, T=\{bC: C\in T \}$ and
$bB=\{bT: T\in B\}$.

\medskip

If the origin $0$ belongs to the interior                   
of a cell, the first cell $C(T)^1$ in the enumeration of $T$ is the one
containing $0$. In this case $C(bT)^1=bC(T)^1$ for $b\neq 0$.

\medskip

We fix a polytope with nonempty interior $W\subset \RR^\ell$,  
and call it a {\it window}.
A tessellation in $W$ is a locally finite countable 
covering of $W$ by polytopes with disjoint
interiors. Let $\T_W$ be the space of tessellations of $W$.
By compactness, each $R\in \T_W$ has a finite number 
of cells $|R|$. The trivial tessellation is $R=\{W\}$ in $\T_W$, and it  
has the boundary $\partial W$. 

\medskip

Let $T\in \T$ and $U\subseteq \RR^\ell$ be a nonempty set such that 
$U=\Cl(\Int U)$. We define the restriction of $T$ to $U$ by,
$$
T\wedge U=\{C\cap U: C\in T, \, \Int (C\cap U) \neq \emptyset\}.
$$
When ${\vec T}=(T_l: l\in L)$ is a family of tessellations we put
${\vec T}\wedge U=(T_l\wedge U: l\in L)$.
Let $W$ be a window and $T\in \T$. We have $T\wedge W \in \T_W$.
Let $W, W'$ be two windows such that $W\subset \Int W'$, then 
every $Q\in \T_{W'}$ defines a tessellation $Q\wedge W\in \T_{W}$. 

\subsection{Measurability considerations}

The family $\F$ of closed sets  of $\RR^\ell$ endowed with
the Fell topology is a metrizable compact
Hausdorff space, see Chapter $12$ in \cite{sw}. 
Let $\F'=\F\setminus \{\emptyset\}$ and
$\F(\F')$ be the family of closed sets of $\F'$ 
endowed with the Fell topology and its associated 
Borel $\sigma-$field ${\cal B}(\F(\F'))$.
The family $\K'$ of nonempty compact convex sets is a Borel set
in $\F'$, that is $\K'\in {\cal B}(\F')$ 
(see Theorem 2.4.2 in \cite{sw}).
On the other hand, since a tessellation $T\in \T$ is a countable 
collection of polytopes,
it is an element of $\F(\F')$. In
Lemma $10.1.2.$ in \cite{sw} it is shown that $\T\in {\cal B}(\F(\F'))$.

\medskip

The space of boundaries of tessellations is a subset of $\F'$ 
and it is endowed with the trace of the Fell topology
and the Borel $\sigma-$field. The topological and measurable 
structures are preserved when representing a tessellation by its boundary, 
in particular every sequence $(T_n: n\in \NN)\in \T^\NN$ and $T\in \T$ 
satisfy: $T_n\to T \Leftrightarrow \partial T_n\to \partial T$.

\medskip

Let $\F_W$ be the family of closed subsets of $W$ and 
$\F'_W=\F_W\setminus \{\emptyset\}$. The set
$\F(\F'_W)$ is endowed with the Fell topology and its 
associated Borel $\sigma-$field. We have $\T_W\in {\cal B}(\F(\F'_W))$.
The $\sigma-$field ${\cal B}(\T_W)$ will be identified with the sub-$\sigma$ 
field ${\cal B}(\T)\wedge W$ of ${\cal B}(\T)$, defined by
$$
{\cal B}(\T)\wedge W:=
\{B=\{T\in \T: T\wedge W\in B_W\}: B_W\in {\cal B}(\T_W)\}\,.
$$
Let $b>0$. Take $B_W\in {\cal B}(\T_W)$. For $Q\in B_W$ we have
$bQ=(bC: C\in Q)\in \T_{bW}$, by definition 
$b B_W=\{bQ: Q\in B_W\}$, so $b B_W\in {\cal B}(\T_{bW})$.
Since $b B=\{bT: t\in B\}$ for $B\in {\cal B}(\T)$, we get
\begin{equation}
\label{amp1}
B\in {\cal B}(\T)\wedge W \, \Rightarrow bB\in {\cal B}(\T)\wedge b W.
\end{equation} 

Since $\T\in {\cal B}(\F(\F'))$,
for any probability measure $\nu$ 
on $(\T, {\cal B}(\T))$ the completed
probability space $(\T, {\cal B}(\T), \nu)$ is Lebesgue.
An analogous statement holds for $\T_W$.

\medskip

Let $(W_l: l\in \NN)$ be a strictly increasing sequence of windows
such that $W_l\subset \Int W_{l+1}$ and 
$W_l\nearrow \RR^\ell$ as $l\nearrow \infty$. 
We have 
\begin{equation}
\label{contseq}
\forall (R_k: k\in \NN)\subset \T, R\in \T: \;\;
R_k\to R  \Leftrightarrow \forall \, l\in \NN:
R_k\wedge W_l\to R\wedge W_l\,.
\end{equation}
and so ${\cal B}(\T)\wedge W_l\nearrow {\cal B}(\T)$ as $l\nearrow \infty$.

\subsection{ The STIT tessellation process }
\label{Sub1.2}
Let us recall the construction of the STIT 
tessellation process $Y=(Y_t: t > 0)$ 
done in \cite{nw}, \cite{mnw}. This is a 
Markov processes whose marginals $Y_t$ take values in $\T$. 
The law of the STIT process $Y$ only depends on a (non-zero)
locally finite and translation invariant measure $\Lambda$ on the 
space of hyperplanes $\hH$ in $\RR^\ell$. 

\medskip
 
Let $S^{\ell-1}$ be the set of unit vectors in $\RR^{\ell-1}$
and ${\widetilde S}^{\ell-1}=S^{\ell-1}/\equiv\,$ be the set of 
equivalence classes for the relation $u\equiv-u$. 
Each hyperplane $h\in \hH$ can be represented by an element 
in $\RR\times {\widetilde S}^{\ell-1}$, expressing the signed distance from the origin and the orthonormal direction of $h$. The image of
$\Lambda$ under this representation is $\lambda\otimes \kappa$,
where $\lambda$ is the Lebesgue measure 
on $\RR$ and $\kappa$ is a finite measure
on ${\widetilde S}^{\ell-1}$, see \cite{sw} Section $4.4$ 
and Theorem $4.4.1.$  From locally finiteness it follows
\begin{equation}
\label{boundxx}
\Lambda([B])\!< \!\infty \; \forall \, B 
\hbox{ bounded in } {\cal B}(\RR^\ell),
\, \hbox{ where } [B]\!=\! \{H\!\in \!{\cal H}: 
H\cap B\neq \emptyset \}. 
\end{equation}
It is assumed that the linear space generated by the support 
of $\kappa$ is $\RR^\ell$, this is written
\begin{equation}
\label{completesup}
\langle \hbox{ Support }\, \kappa \; \rangle=\RR^\ell\,.
\end{equation}
Let $W$ be a window. From (\ref{boundxx}) we get 
$0<\Lambda([W])<\infty$. 
The translation invariance of $\Lambda$ yields
\begin{equation}
\label{homogxx}
\Lambda ([cW]) = c\, \Lambda ([W]) \quad \mbox{ for all } c>0
\end{equation}
(see e.g. \cite{sw}, Theorem $4.4.1.$).
Denote by 
$\Lambda^W(\bullet)=\Lambda ([W])^{-1}\Lambda(\bullet\; \cap [W])$ 
the normalized probability measure on the set of hyperplanes intersecting $W$.

\medskip

The restriction of $Y$ to a window $W$ is a pure jump Markov 
process $Y^W=(Y^W_t: t\ge 0)$ 
whose marginals $Y^W_t$ take values in $\T_W$. To describe its construction, let
$(h_{n,m}: n\in \ZZ_+, m\in \NN)$ and 
$(e_{n,m}: n\in \ZZ_+, m\in \NN)$
be two independent families of independent random variables with 
distributions $h_{n,m}\sim \Lambda^W$ and $e_{n,m}\sim 
\hbox{Exponential}(1)$. By an inductive procedure we will 
define an increasing sequence of random times 
$(S_n: n\in \ZZ_+)$ and a
sequence of random tessellations 
$(Y^W_{S_n}: n\in \ZZ_+)$ 
with starting points $S_0=0$ and $Y^W_0= \{W\}$ as follows:
Let $\{C^1_n,...,C^{n+1}_n\}$ be the cells of $Y^W_{S_n}$, we put
\begin{eqnarray*}
&{}& S_{n+1}=S_n+e(Y^W_{S_n}) \hbox{ where }
e(Y^W_{S_n})=\min\{e_{n,l}/\Lambda([C^l_n]): l=1,...,n+1\} \hbox{ and }\\
&{}& Y^W_{S_{n+1}} \hbox{ defined by the cells }
\{C^l_n: l\neq l^* \} \cup \{ C'_1, C'_2\},
\end{eqnarray*}
where $C'_1, C'_2$ is the partition of $C^{l^*}_n$ by the hyperplane
$h_{n,m}$, being $m$ the first index
such that $h_{n,m}\in [C_n^{l^*}]$. We note that the index $l^*$ such 
that $e_{n,l^*}/\Lambda([C_n^{l^*}])=e(Y^W_{S_n})$ 
is a.e. uniquely defined. 
It can be shown that $S_n\to \infty$ as $n\to \infty$.

\medskip

We define the process $Y^W$ by
\begin{equation}
\label{cadlagY} 
Y^W_t=Y^W_{S_n}, \;\, t\in [S_n,S_{n+1}).
\end{equation}
This is a pure jump Markov process. This construction
yields a law consistent with respect to the family of windows:
$W'\subseteq W$ implies $Y^{W}\wedge W' \sim Y^{W'}$. 

\medskip

In \cite{mnw} it was shown that there is a well-defined 
Markov process $Y=(Y_t: t>0)$, this is a STIT tessellation process,
with marginals $Y_t$ taking values in $\T$ 
and that satisfies $Y_t\wedge W\sim Y_t^W$ for all window $W$ and $t>0$. 
From (\ref{cadlagY}) $Y\wedge W$ is a pure jump Markov process and 
so with c\`adl\`ag trajectories and
from (\ref{contseq}) we get that also $Y$ has
c\`adl\`ag trajectories. Then, the trajectories of
$Y\wedge W$ belong to $D_{\T_W}(\RR_+)$ and the trajectories
of $Y$ are in the metric separable space $D_{\T}(\RR_+)$.
Since the closure cl ${\T}$ in $\F(\F')$ is a Polish space, then
we can assume that the trajectories 
of $Y$ take values in the Polish
space $D_{cl {\T}}(\RR_+)$.

\medskip

From the construction and since $S_1$ is exponentially distributed 
with parameter $\Lambda ([W])$ we get
$$
\PP(\partial(Y_t \wedge W )\cap \Int W \!=\!\emptyset )\!=\!
\PP(Y_t \wedge W \!=\!\{W\})\!=\!
\PP(Y_t \wedge W \!=\!Y_0 \wedge W)\!=\!e^{-t \Lambda([W])}.
$$
For $t>0$ let $\xi^t$ and $\xi_W^t$ be the marginal distributions
of $Y_t$ and $Y_t\wedge W$, that is
$$
\xi^t(B)=\PP(Y_t\in B) \; \forall B\in {\cal B}(\T)
\hbox{ and } \xi_W^t(D)=\PP(Y_t\wedge W\in D) \; \forall D\in {\cal B}(\T_W).
$$
We have $\xi_W^t(\{W\})=e^{-t \Lambda([W])}>0$, so $\{W\}$ is an 
atom. In \cite{mn1} is was shown that $\{W\}$ is the unique 
atom of $\xi_W^t$, which implies that $\xi^t$ is non-atomic.

\medskip

Moreover, it was shown in \cite{nw03} that the distribution 
of the zero cell of the STIT tessellation $Y_1$ is identical to 
the distribution of the zero cell $C(P)^1$ of a Poisson hyperplane 
tessellation with intensity measure $\Lambda$. This implies for all 
compact convex $K\subset \RR^\ell$ with $0\in \Int K$
\begin{equation} 
\label{poisscell}
\PP(\partial Y_1\cap K =\emptyset)=\PP(C(P)^1 \supset K)=e^{-\Lambda([K])}.
\end{equation}

The following scaling property, which is used to state the renormalization
in time and space, was shown in \cite{nw}, Lemma $5$,
\begin{equation}
\label{homothet}
\forall \, t>0\,:\;\; t Y_t \sim Y_1 \,.
\end{equation}

\subsection{Independent increments relation}
\label{indincr}

Let $T\in \T$ be a tessellation and  
${\vec R}=(R_k: k\in \NN)\in \T^{\NN}$ be a sequence of tessellations.
We define the tessellation $T\boxplus {\vec R}$ (also referred
as iteration or nesting) by the set of cells 
resulting from the restriction
of the tessellation $R_k$ to the cell $C(T)^k$:
\begin{equation}\label{defiteration}
T\boxplus {\vec R}\!=\!\{ C(T)^k\!\cap \!C(R_k)^l: \,
k\!\in \!\NN,\, l\!\in \!\NN,\,
\Int(C(T)^k\!\cap \!C(R_k)^l)\!\neq \!\emptyset \} .
\end{equation}
Assume $T$ is such that the origin is in the interior of one of its
cells. Then, $R_1$
is the tessellation that by this operation is restricted to the cell $C(T)^1$ 
containing the origin.

\medskip

Let ${\vec Y}'=({Y'}^m: m\in \NN)$ be a sequence of independent 
copies of $Y$, that is ${Y'}^m\sim Y$, and also
independent of $Y$. Let ${\vec Y}'_s=({Y_s'}^m: m\in \NN)$ for $s>0$.
From the construction of $Y$ we have the following
relation was first stated in Lemma $2$ in \cite{nw},
\begin{equation}
\label{iterate22a}
Y_{t+s} \sim Y_t\boxplus {\vec Y}'_s \ \mbox{ for all }t,s>0\,.
\end{equation}
The construction done in  \cite{nw} for proving this result
also allows to show the following relation stated in \cite{mn1}. 
Let ${\vec Y}^{'(i)}$, $i=1,\ldots,j$ be a sequence of
$j$ independent copies of ${\vec Y}'$ and also independent of $Y$.
Then, for all $0<s_1<...<s_j$ and all $t>0$ we have
\begin{equation}
\label{iterate22}
(Y_t,Y_{t+s_1},...,Y_{t+s_j}) \sim
(Y_t, Y_t\boxplus {\vec Y}^{'(1)}_{s_1},...,
(((Y_t\boxplus {\vec Y}^{'(1)}_{s_1})\boxplus....)
\boxplus{\vec Y}^{'(j)}_{s_j-s_{j-1}})).
\end{equation}

\section{The renormalized stationary process}
\subsection{Properties of the renormalized process}

Fix $a>1$ and define the renormalized process 
$\Z=(\Z_s: s\in \RR)$ by
$\Z_s=a^s Y_{a^s}$ for $s\in \RR$.
Note that $\Z_0=Y_1$. Since $Y$ is a Markov process, so is $\Z$.
From (\ref{homothet}) all $1$-dimensional distributions of
$Z$ are identical. In Theorem 1.1. in \cite{mn1} it was 
shown that $\Z$ is a stationary Markov process. 
The process $\Z$ inherits  c\`adl\`ag trajectories from $Y$,
so it takes values in $D_{\T}(\RR)$.
Let $\mu^\Z$ be the law of $\Z$ on $D_{\T}(\RR)$, then
$(D_{\T}(\RR), \mu^\Z, (\sigma^t_\T))$ is a shift flow.
The discrete process $\Z^d=(\Z_n: n\in \ZZ)$ is also Markov and
its law on $\T^\ZZ$ is denoted by $\mu^{\Z^d}$.

\medskip

Let $\Z\wedge W=(\Z_s\wedge W: s\in \RR)$
and $\Z^d\wedge W=(\Z_n\wedge W: n\in \ZZ)$
be the continuous and discrete process restricted to the window $W$.
Their laws are respectively denoted by $\mu^\Z_W$ 
and $\mu^{\Z^d}_W$. 
The mappings $\Z\to \Z\wedge W$ and $\Z^d\to \Z^d\wedge W$  
are factor maps. In Theorem 1.2 in \cite{mn1} it was stated that 
$\Z\wedge W$, and so also $\Z^d\wedge W$, is a mixing Markov 
stationary process.
But the Markov property of $\Z\wedge W$ (which is not a simply 
consequence of the Markovianness of $\Z$) 
was not shown in detail. Due to the
central role it plays in our main results by Lemma 
\ref{lemma4}, and for completeness, we give a proof of it here.

\medskip

\begin{lemma}\label{lemma3}
The restricted processes $\Z\wedge W$ and $\Z^d$ are Markov processes. 
\end{lemma}

\begin{proof}
It suffices to show that $\Z\wedge W$ is Markov.
Fix $B\in {\cal B}(\T_W)$. For $b>0$, from 
(\ref{amp1})
we have $b^{-1}B\in {\cal B}(\T_{b^{-1}W})$ and for all
$T\in \T$, the relation $bT\wedge W\in B$ holds if and only if 
$T\wedge b^{-1}W\in b^{-1}B$.
So, for $h>0$ it holds
\begin{eqnarray*}
&{}& \PP(\Z_{t+h}\wedge W \in B \, | \, Z_s\wedge W, s\le t)\\
&=&
\PP(a^{t+h}Y_{a^{t+h}}\wedge W \in B \, | \, a^s Y_{a^s}\wedge W, 
s\le t)\\
&=&
\PP(Y_{a^{t+h}}\wedge a^{-(t+h)}W \in a^{-(t+h)}B \, | 
\, Y_{a^s}\wedge a^{-s} W, s\le t)\\
&=&
\PP(Y_{a^{t+h}}\wedge a^{-(t+h)}W \in a^{-(t+h)} B \, | \, 
Y_{a^s}\wedge a^{-(t+h)}W, s\le t)\\
&=& \PP(Y_{a^{t+h}}\wedge a^{-(t+h)}W \in a^{-(t+h)} B \, | \,
Y_{a^t}\wedge {a^{-t}}W))\\
&=&\PP(\Z_{t+h}\wedge W \in B \, | \, \Z_t\wedge W)
\end{eqnarray*}
where in the third and fourth equalities we use that
$(Y_u\wedge a^{-(t+h)}W: u>0)$ is a Markov process and that
$a^{t+h}>a^s$ and $a^{-(t+h)}\le a^{-s}$ for all $s\le t$ and $h>0$,
because $a>1$.
\end{proof}

\bigskip

Let ${\vec Y}'=({Y'}^m: m\in \NN)$ be a sequence of independent
copies of $Y$, and independent of $Y$. From property (\ref{iterate22}) 
it follows 
$$
\Z_{n+1}\sim a \Z_{n} \boxplus a^{n+1}{\vec Y}'_{a^{n+1}-a^{n}}\,.
$$
Since
$a^{n+1}{\vec Y}'_{a^{n+1}-a^{n}}=\tfrac{a}{a-1}(a^{n}(a-1)
{\vec Y}'_{a^{n}(a-1)})$ we get from (\ref{homothet}),
\begin{equation}
\label{recrel}
(\Z_{n},\Z_{n+1})\sim (\Z_{n}, a \Z_{n} 
\boxplus \tfrac{a}{a-1}{\vec Y}'_1)\,.
\end{equation}
Let $({\vec Y}^{'(i)}_1: i\ge 0)$ be independent copies of
${\vec Y}^{'}_1$. A simple recurrence on (\ref{recrel}) and
using (\ref{iterate22a}) yields the following formula for the finite-dimensional distributions of $\Z^d$
\begin{equation}
\label{recrege}
(\Z_{n+i}:0\le i\le k)\sim \left(a^{i} \Z_{n}  \boxplus_{j=1}^{i}
\tfrac{a^{i+1-j}}{a-1}{\vec Y}^{'(j)}_1: 0\le i\le k \right)\,,
\end{equation}
for $n\in \ZZ$ and $k\ge 0$.
Recall that $M\boxplus_{i=1}^{k} {\vec M}^{'(i)}$ is an abbreviation
for \\
 $\left( \ldots \left( M \boxplus {\vec M}^{'(1)} \right)
\boxplus \ldots \right) \boxplus  {\vec M}^{'(k)}  $, where $M$ is a
tessellation and ${\vec M}^{'(i)}$ a sequence of tessellations.

Let us now consider the joint distribution of the zero cell process of $\Z^d$, denoted by $\Gamma^d=(\Gamma_n: n\in \ZZ)$. Let $\C^{'(i)}_1$ denote the zero cell of the first element of the sequence ${\vec Y}^{'(i)}_1$. Thus $(\C^{'(i)}_1 : i\ge 0)$ is sequence of independent and identically distributed zero cells.
Then (\ref{recrege}) and (\ref{defiteration}) yield
\begin{equation}
\label{recregezero}
(\Gamma_{n+i}:0\le i\le k)\sim \left( \left(a^{i} \Gamma_{n}\right) \cap \bigcap_{j=1}^{i}
\tfrac{a^{i+1-j}}{a-1}{\C}^{'(j)}_1: 0\le i\le k \right)\,,
\end{equation}
for $n\in \ZZ$ and $k\ge 0$.

\section{Regenerative structure of the stationary zero cell process}

\subsection{Stationary renewal sequences}
Let $(V_n: n\in \ZZ)$ be a stationary $0-1$ valued sequence. 
We define the vector of transition probabilities $\vq=(q_n: n\in \NN)$ by
\begin{equation}
\label{hoy3}
\forall n\in \NN: \quad q_n=\PP(V_n=1 \, | \, V_0=1)
=\PP(V_{n+i}=1 \, | \, V_i=1),
\end{equation}
the last equality follows from stationarity. For $m<n$ we get
$$
\PP(V_n=1 \, | \, V_m=1)\PP(V_m=1)=\PP(V_m=1 \, | \, V_n=1)\PP(V_n=1).
$$
Since $\PP(V_m=1)= \PP(V_n=1)$ we deduce
\begin{equation}
\label{hoy10}
\PP(V_n=1 \, | \, V_m=1)=\PP(V_m=1 \, | \, V_n=1)=q_{n-m}.
\end{equation}

We shall assume that the process {\it regenerates} at the $1-$values. 
More precisely, we assume that (see \cite{sro}, Chapter 3, Section 3.7): 
\begin{eqnarray}
\nonumber
&{}& \forall\, n\in \NN, \, r\ge 1,\; i_0<i_1<i_2...<i_{n},\; 
(a_k:k=1,..,r)\in \{0,1\}^r:\\
\nonumber
&{}& \PP(V_{i_{n}+k}=a_k, k=1,..,r \, | \, V_{i_{n}}=1, 
V_{i_{n-1}}=1,..., V_{i_0}=1)\\
\label{hoy0}
&{}&=\PP(V_{i_{n}+k}=a_k, k=1,..,r \, | \, V_{i_n}=1).
\end{eqnarray}
By stationarity, this is equivalent to
\begin{eqnarray*}
&{}& \PP(V_{k}=a_k, k=1,..,r \, | \, V_{0}=1, 
V_{-j_{1}}=1,...,V_{-j_{n}}=1)\\
&{}&=\PP(V_{k}=a_k, k=1,..,r \, | \, V_0=1),
\end{eqnarray*}
for all $0<j_{1}<...<j_{n}$. Then, from this regeneration property and  (\ref{hoy3}),
\begin{equation}
\label{hoy1}
\PP(V_{i_{0}}=1, V_{i_{1}}=1,...,V_{i_{n-1}}=1 ,V_{i_{n}}=1)
=\left(\prod_{k=0}^{n-1} q_{i_{k+1}-i_k}\right)\PP(V_{i_0}=1).
\end{equation}
For a deeper treatment of regenerative process see \cite{asm} 
Chapter VI.

\medskip

Let us consider the random set
$$
\V^*=\{n\in \ZZ: V_n=1\}.
$$
It is stationary because $\V$ is an stationary 
sequence, and so for all 
$a\in \ZZ$ we have $\V^*\sim \V^*+a$ or equivalently
$\{n\in \V^*: n\ge 0\}\sim \{n\textcolor{red}{-a}\in \V^*: n\ge a\}$.

\medskip

Let us define the interarrival distribution ${\vp}=(p_n: n\in \NN)$ 
of $\V^*$:
\begin{equation}
\label{inter}
\forall n\in \NN:\quad
p_n=\PP(V_n=1, V_l=0, \;\,  0<l<n \; | \, V_0=1)
\end{equation}
By stationarity $p_n=\PP(V_{n+k}=1, V_{l+k}=0, \;\, 0<l<n \; | \, V_k=1)$
for all $k\in \ZZ$.

\medskip

The random set $\V^*$ is an stationary renewal set with
interarrival distribution ${\vp}$. Stationarity implies that the mean 
recurrence time is finite, this is
$\rho=\sum_{n\in \NN} n \, p_n<\infty$, and we have:
\begin{equation}
\label{eqprob}
\forall a\in \ZZ: \quad \PP(a\in \V^*)=\PP(0\in \V^*)=
\PP(V_0=1)=\rho^{-1}.
\end{equation}
We enumerate the elements of this set, i.e.  we put $\V^*=\{V^*_i : i\in \ZZ\}$ by 
imposing: $V^*_i<V^*_{i+1}$ for all $i\in \ZZ$ and
$V^*_0=\inf\{n\in \NN: n\in \V^*\}$. Then,
$$
\forall i\ge 0, n>0: \quad p_n=\PP(V^*_{i+1}-V^*_i=n). 
$$
The stationarity property is equivalent to,
\begin{equation}
\label{eqprob1}
\forall k\in \NN: \quad \PP(V^*_0-V^*_{-1}=k)=\rho^{-1} k \, p_k.
\end{equation}
In this case we have,
\begin{equation}
\label{eqprob2}
\forall k\in \ZZ_+: \quad \PP(V^*_0=k)=\rho^{-1} \sum_{m>k} p_m.
\end{equation}
For all these results on stationary renewal sets see 
\cite{tl}, Chapter II.

Note that the reverse process
$-\V^*$ is also a stationary renewal 
sequence with the same interarrival law as $\V^*$.

\medskip

Let us describe the interarrival distribution $\vp$ of the
renewal set $\V^*$ in terms of $\vq$. We apply the
inclusion-exclusion principle and (\ref{hoy1}).

\medskip

\begin{proposition}
\label{inclexcl}
Let $n\in \NN$ and $I\subseteq \{ 1,...,n-1\}$. Let $|I|$
be the number of elements of $I$ and denote
its elements by $i^I_1<..<i^I_{|I|}$. Further,
put $i^I_0=0$ and $i^I_{|I|+1}=n$. Then
\begin{eqnarray*}
p_n &{=}&\PP(V_0=1, V_{k}=0 \; \forall 0<k<n\; | \, V_n=1)\\
&=&\rho \, \sum_{I\subseteq \{ 1,...,n-1\}} (-1)^{|I|}
\left(\prod_{k=0}^{|I|} q_{i^I_{k+1}-i^I_k}\right),
\end{eqnarray*}
where the sum includes the summand for $I=\emptyset$.
\end{proposition}
  
\begin{proof}
By (\ref{eqprob}) we get,
\begin{eqnarray*}
p_n&=&\PP(V_n=1, V_k=0,\; \forall 0<k<n \; | \, V_0=1)\\
&=&\PP(V_n=1, V_k=0, \; \forall 0<k<n, V_0=1) \rho.
\end{eqnarray*}
Let us express $\rho^{-1}\, p_n=\PP(V_0=1, V_k=0, \; 
\forall 0<k<n, V_n=1)$ in terms of $\vq$.
For $k=0,\ldots ,n$ define the event $A_k = \{ V_k=1\}$. Then
\begin{eqnarray*}
\rho^{-1} \, p_n
&=& \PP \left( A_0 \cap \bigcap_{k=1}^{n-1} A_k^c \cap A_n \right)
= \PP \left( A_0 \cap \left(\bigcup_{k=1}^{n-1} A_k \right) ^c
\cap A_n \right) \\
&=& \PP \left( A_0 \cap  A_n \right) -
\PP \left( A_0 \cap \left(\bigcup_{k=1}^{n-1} A_k \right) 
\cap A_n \right). 
\end{eqnarray*}
Hence, by using the inclusion-exclusion principle and (\ref{hoy1}) we 
obtain,
\begin{eqnarray*}
\rho^{-1}\, p_n&=& \PP \left( A_0 \cap  A_n \right) -
\sum_{I\subseteq \{1,...,n-1\}, I\neq \emptyset} (-1)^{|I|+1}\;
\PP\left(A_0 \cap \bigcap_{j\in I} A_{j} \cap A_n \right)
\\
&=& \sum_{I\subseteq \{ 1,...,n-1\}} (-1)^{|I|}\;  
\PP \left( A_0 \cap \bigcap_{j\in I}  A_{j} \cap A_n \right)\\
&=& \sum_{I\subseteq \{ 1,...,n-1\}} (-1)^{|I|} \left(\prod_{k=0}^{|I|} 
q_{i^I_{k+1}-i^I_k}\right).
\end{eqnarray*}
So, the result is shown.
\end{proof}

\subsection{Regenerative properties of the stationary zero cell process}

As already mentioned in the introduction, denote by $\C_t$ the zero cell of the PHT $X_t$ or the STIT $Y_t$, respectively. The process $\C=(\C_t:\, t > 0)$ is 
well-defined a.e. 
The process $\Gamma=(\Gamma_t:=a^t \C_{a^t}: t\in \RR)$
is the zero cell process of $\Z$ and 
$\Gamma^d=(\Gamma_n: n\in \ZZ)$ is the zero cell sequence
of $\Z^d$. Both process, $\Gamma$ and $\Gamma^d$ are stationary.

\medskip

Let us construct a regenerative sequence for 
the discrete stationary zero cell sequence.
Let $K$ be a compact and convex set containing $0$ in its interior, 
and consider the random sequence $\V^K=(V^K_n: n\in \ZZ)$ 
of $0-1$ valued random variables 
\begin{equation}
\label{defV}
V^K_n=\1_{\{\Gamma_n\supset K\}}, n\in \ZZ.
\end{equation}

We have the equality of events
\begin{equation}
\label{hoy33}
\{V^K_n=1\}=\{\Gamma_n\supset K\}=\{\partial {\Z_n}\cap K=\emptyset\} .
\end{equation}
The random sequence $\V^K=(V^K_n: n\in \ZZ)$ inherits stationarity from 
$\Z^d\wedge K$. From (\ref{poisscell}) we have 
\begin{equation}
\label{trproze}
\forall n\in \ZZ:\quad \PP(V^K_n=1)=\PP(\Gamma_n\supset K)=e^{-\Lambda([K])}.
\end{equation}

\begin{lemma}
\label{lemma4}
The $0-1$ stationary sequence  $\V^K=(V^K_n: n\in \ZZ)$ satisfies
the regenerative property (\ref{hoy0}).
\end{lemma}

\begin{proof}
Let $i_0<i_0<i_1...<i_n$, and two disjoint finite sets $I$ and $J$ be
included in $\NN$. 
Therefore, we have 
\begin{eqnarray*}
\nonumber
&{}& \PP(V^K_{i_n+k}=1 \; k\in I, V^K_{i_n+k}=0 \; k\in J 
\, | \, V^K_{i_{n}}=1,V^K_{i_{n-1}}=1...,V^K_{i_0}=1)\\
&{}& =\PP\left(\Z_{i_{n}+k}\wedge \!K \!=\! \{K\}\; k\in I,  
\Z_{i_{n}+k}\wedge \!K  \!\neq \!\{K\}\; k\in J \right.\\
 &{}& \quad\quad  \left. \, | \, \Z_{i_{n}}\wedge K\!=\!\{K\},
\Z_{i_{n-1}}\wedge\,K \!=\!\{K\},...,\Z_{i_0}\wedge K\!=\! 
\{K\}\right)\\
&{}&=\PP\left(\Z_{i_{n}+k}\wedge \!W\!=\!\{K\} \; k\in I,  
\Z_{i_{n}+k}\wedge \!K \!\neq \!\{K\}\; k\in J \, | \, 
    \Z_{i_{n}}\wedge \!K=\{K\}\right)\\
&{}& =\PP(V^K_{i_n+k}=1\; k\in I, V^K_{i_n+k}=0 \; k\in J \, | \, 
V^K_{i_{n}}=1).
\end{eqnarray*}
In the the second equality we use that $\Z^d\wedge K$ is a 
Markov process, see Lemma \ref{lemma3}. Thus the result is shown.
\end{proof}

Let us compute the transition probability vector 
$\vq^{\,K}=(q^K_n: n\in \NN)$ of $\V^K$, which is given by (\ref{hoy3}),
$$
\forall n\in \NN: \quad q^K_n=\PP(V^K_n=1 \, | \, V^K_0=1)
=\PP(V^K_{n+i}=1 \, | \, V^K_i=1),
$$

\begin{lemma} 
\label{lemma2}
We have
\begin{equation}
\label{qzero}
\forall n\in \NN: \quad q^K_n=e^{-(1-a^{-n})\Lambda([K])}.
\end{equation}
\end{lemma}

\begin{proof}
From (\ref{recregezero}) we get
\begin{eqnarray}   
\label{eqcap}
&{}&\PP({\Gamma_n}\supset K)\\
\nonumber
&{}& =\PP\left(\{ {\Gamma_0}\supset a^{-n} K\}\cap 
\{ {\C}^{'(j)}_1 \supset (a-1) a^{-j} K, 
j=1,..,n  \} \right).
\end{eqnarray}
From
$$
\{ {\Gamma_0}\supset K,  {\Gamma_0}\supset
a^{-n} K\}=\{ {\Gamma_0}\supset K\}, 
$$ 
we deduce
$$
\PP(V^K_n=1, V^K_0=1)= \PP\left( \Gamma_0\supset  K,\;
 {\C}^{'(j)}_1 \supset (a-1) a^{-j} K, 
j=1,..,n \right).
$$
The $n+1$ random variables $\Gamma_0$, $\{\C^{'(j)}_1: 
j=1,..,n\}$ are mutually independent, hence
\begin{eqnarray*}
&{}& \PP(V^K_n=1, V^K_0=1)\\
&=&\PP\left( {\Gamma_0}\supset K, 
 {\C}^{'(j)}_1 \supset (a-1) a^{-j} )
, j=1,..,n\right)\\
&=&
\PP(\Gamma_0\supset K \, | \, 
{\Gamma_0}\supset a^{-n} K)\\
&{}&\cdot \PP\left( {\Gamma_0}\supset a^{-n} K, 
 {\C}^{'(j)}_1 \supset (a-1) a^{-j} K, 
j=1,..,n\right).
\end{eqnarray*}
We use (\ref{homogxx}) to get
$$
\PP({\Gamma_0}\supset K \, | \,  {\Gamma_0}\supset 
a^{-n} K)=e^{-(1-a^{-n})\Lambda([K])}.
$$
Since
\begin{eqnarray*}
&{}&\PP\left({\Gamma_0}\supset a^{-n} K, 
{\C}^{'(j)}_1 \supset (a-1) a^{-j} K), 
j=1,..,n\right)\\
&{}&=\PP( {\Gamma_n}\supset K)=
\PP(V^K_n=1)=\PP(V^K_0=1),
\end{eqnarray*}
we get the result
$$
\PP(V^K_n=1 \, | \, V^K_0=1)=e^{-(1-a^{-n})\Lambda([K])}.
$$
\end{proof}

Therefore,
\begin{eqnarray*}
&{}& \PP(V^K_{i_{0}}=1, V^K_{i_{1}}=1,...,V^K_{i_{n}}=1)=
\PP(V^K_{i_0}=1)\cdot \left(\prod_{l=1}^{n}
q_{i_l-i_{l-1}}\right)\\
&{}&\; = e^{-\Lambda([K])}\cdot \left(\prod_{l=1}^{n}
e^{-(1-a^{-(i_l-i_{l-1})})\Lambda([K])}\right)
=e^{(-(n+1)+\sum_{l=1}^n a^{-(i_l-i_{l-1}})\Lambda([K])}.
\end{eqnarray*}

Since the process $\V^K$ satisfies the regenerative property, we can 
associate to it the stationary regenerative set $\V^{K*}$
The interarrival distribution $\vp=(p_n: n\in \NN)$ of $\V^{K*}$ 
can be obtained from the transition 
probability vector $\vq$ as in Proposition \ref{inclexcl}. From
(\ref{eqprob}) we have
\begin{equation}
\label{intl1}
\PP(0\in \V^{K*})=\PP(V_0=1)=e^{-\Lambda([K])},
\end{equation}
and so from (\ref{trproze}) we deduce that 
the mean recurrence time is
\begin{equation}
\label{intl2}
\rho_W=\sum_{n\in \NN} n p_n=e^{\Lambda([K])}.
\end{equation}

\subsubsection{Probabilistic relations between $\V^{a\,K*}$
and $\V^{K*}$} 

Let us study the conditional joint 
distribution $\PP(\V^{a\,K*} \,|  \, V^{K*})$. We start with some preliminary considerations.

If $K \subset K'$ are two compact convex sets with nonempty interior then
$\Gamma_i\supset K'$ implies $\Gamma_i\supset K$ and hence
$\V^{K'\,*}\subseteq V^{K*}$ In particular, for $a>1$ we have
$\V^{a\, K*}\subseteq V^{K*}$.

\medskip

Further, from (\ref{intl1}) and (\ref{intl2}) we have
\begin{equation*}
\PP(0\in \V^{a\,K*})=\PP(V^{a\,K}_0=1)=e^{-a\Lambda([K])} \hbox{ and }
\rho_{a\, K}=e^{a\Lambda([K])}.
\end{equation*}

A straightforward calculation yields
\begin{eqnarray}\label{intl3}
\PP(V^{a\, K}_n=1 \, | \, V^{K}_n=1)&=&
\frac{\PP(V^{a\, K}_n=1)}{\PP(V^{K}_n=1)}
=\frac{\PP( \Gamma_n\supset a \, K)}{\PP(\Gamma_n\supset K)}
\label{formtr}
=e^{-(a-1)\Lambda([K])}. 
\end{eqnarray}
Hence
$$
\rho_{a \,K}^{-1}= e^{-(a-1)\Lambda([K])}\rho_K^{-1}.
$$
In general, for $i>0$ we have
\begin{equation}
\label{relxx}
\PP(V^{a^i\, K}_n=1 \, | \, V^{K}_n=1)=e^{-(a^i-1)\Lambda([K])}
\hbox{ and } \rho_{a^i \,K}^{-1}= e^{-(a^i-1)\Lambda([K])}\rho_K^{-1}.
\end{equation}

Moreover, from (\ref{eqcap}) we obtain for all $n\in \ZZ$
\begin{eqnarray*}
\nonumber
&{}&\PP(V_n^{a\,K}=1, V^{K}_{n-1}=1, V^{K}_n=1)=
\PP(V_n^{a\,K}=1, V^{K}_{n-1}=1)\\
\nonumber
&{}& =\PP\left(\Gamma_{n-1}\supset K,\,
 {\C}^{'(1)}_1 \supset ((a\!-\!1) a^{-1} aK)\right)\\
\label{evrx1}
&{}&=\PP(\Gamma_{n-1}\supset K) e^{-(a-1)\Lambda([K])}.
\end{eqnarray*}
Analogously
\begin{eqnarray*}
\nonumber
\PP(V^{K}_{n-1}=1, V^{K}_n=1)
&=&\PP\left( {\Gamma_{n-1}}\supset K,\,
 {\C}^{'(1)}_1 \supset (a\!-\!1) a^{-1} K
\right)\\
\label{evrx2}
&=&\PP( {\Gamma_{n-1}}\supset K=\emptyset)e^{-(1-a^{-1})\Lambda([K])}.
\end{eqnarray*}
Hence
\begin{equation}\label{evrx5}
\PP(V_n^{a\,K}=1 \, | \,  V^{K}_n=1, V^{K}_{n-1}=1)=
e^{-(a+a^{-1}-2)\Lambda([K])}.
\end{equation}
Also, it is easy to see that
$$
\left\{ V^{a\, K}_n=1\} \subset  \{ V^{K}_n=1 , V^{K}_{n-1}=1 
\right\} 
$$
which implies that
$$
\PP(V_n^{a\,K}=1 \, | \,  V^{K}_n=1,  V^{K}_{n-1}=0 )=
0.
$$

Hence, the values of 
$\PP(V_n^{a\,K}=1 \, | \,  V^{K}_n=1, V^{K}_{n-1}=1)$ given in (\ref{evrx5}) are the key for the conditional joint distribution $\PP(\V^{a\,K*} \,|  \, V^{K*})$, 

We introduce the
following notation: For $I\subset \ZZ$, by $V^K_I=1$
we mean $V^K_n=1$ for all $n\in I$ and similarly $V^K_I=0$
expresses $V^K_n=0$ for all $n\in I$.

\begin{proposition}
\label{coditionalnew}
Let $I^0=[\alpha,\beta]$ be a finite interval in $\ZZ$ 
containing 
a finite family of disjoint intervals $I=\bigcup_{t=1}^l I_t$  
with $I_t=[\alpha_t,\beta_t]$ in $\ZZ$ that satisfies 
$$
\alpha<\alpha_1, \;\; \beta_l<\beta, \;\; \alpha_t < \beta_t<\alpha_{t+1}-1, \; 
\forall \, t=1,..,l-1.
$$
Let $\wI_t=[\alpha_t+1,\beta_t]$.
and  $J_t\subseteq \wI_t$ for $t=1,..,l$ and  $\wI=\bigcup_{t=1}^l \wI_t$ and $J=\bigcup_{t=1}^l J_t$. 
Then,
\begin{eqnarray}
\nonumber
&{}&
\PP\left(V_J^{a\,K}=1, V_{\wI\setminus J}^{a\,K}=0
\, | \, V^K_I=1, V^K_{I^0\setminus I}=0\right) \\ 
\label{hoy121} 
&{}& =e^{-(a+a^{-1}-2)|J|\Lambda([K])}(1-e^{-(a+a^{-1}-2)
\Lambda([K])})^{|\wI\setminus J|} \; .
\end{eqnarray}
\end{proposition}

\begin{proof}
First note that under the condition $V^K_I=1$ we have $V^K_i=1$ and $V^K_{i-1}=1$ for all $i\in \wI$.
For all $i\in J$ we obtain
\begin{eqnarray}\label{eqcondrec}
&& \PP \left( V_i^{a\,K}=1, V_j^{a\,K}=1, j\in J, j<i,  V_{j}^{a\,K}=0 , j\in\wI\setminus J, j<i
\, | \, V^K_I=1, V^K_{I^0\setminus I}=0 \right) \nonumber \\
&=&  \PP \left( V_i^{a\,K}=1 \, | \,  V_j^{a\,K}=1, j\in J, j<i,  V_{j}^{a\,K}=0 , j\in\wI\setminus J, j<i,
V^K_I=1, V^K_{I^0\setminus I}=0 \right) \nonumber\\
&& \cdot \PP \left( V_j^{a\,K}=1, j\in J, j<i,  V_{j}^{a\,K}=0 , j\in\wI\setminus J, j<i
\, | \, V^K_I=1, V^K_{I^0\setminus I}=0 \right) . 
\end{eqnarray}
Let us consider the first factor in the expression above. 
An application of (\ref{recregezero}) yields
\begin{eqnarray*}
&{}&  \PP \left( V_i^{a\,K}=1 \, | \,  V_j^{a\,K}=1, j\in J, j<i,  V_{j}^{a\,K}=0 , j\in\wI\setminus J, j<i,
V^K_I=1, V^K_{I^0\setminus I}=0 \right) \\
&=& \PP \left( a\Gamma_i \cap \frac{a}{a-1}{\C}^{'(1)}_1\supset {a\,K} \, | \, 
\Gamma_i\supset K, a\Gamma_i \cap \frac{a}{a-1}{\C}^{'(1)}_1\supset K, \right. \\
&& \hspace{2cm} \left. V_j^{a\,K}=1, j\in J, j<i,  V_{j}^{a\,K}=0 , j\in\wI\setminus J, j<i,
V^K_I=1, V^K_{I^0\setminus I}=0 \right) \\
&=& \PP \left( \Gamma_i \supset K, {\C}^{'(1)}_1\supset {(a-1)\,K}\, | \, 
\Gamma_i\supset K,  {\C}^{'(1)}_1\supset \frac{a-1}{a}K, \right. \\
&& \hspace{2cm} \left. V_j^{a\,K}=1, j\in J, j<i,  V_{j}^{a\,K}=0 , j\in\wI\setminus J, j<i,
V^K_I=1, V^K_{I^0\setminus I}=0 \right) \\
&=& \PP \left(  {\C}^{'(1)}_1\supset {(a-1)\,K}\, | \, 
 {\C}^{'(1)}_1\supset \frac{a-1}{a}K,  \right)\\
&=& e^{-(a+a^{-1}-2)\Lambda([K])} .
\end{eqnarray*}
In the last but one equation we used that $ {\C}^{'(1)}_1$ is independent from $\Gamma_i$ as well as from $\{ V_j^{a\,K}=1, j\in J, j<i,  V_{j}^{a\,K}=0 , j\in\wI\setminus J, j<i,
V^K_I=1, V^K_{I^0\setminus I}=0\}$.

Now, a recursive application of equation (\ref{eqcondrec}) yield the proposition.
\end{proof}

\begin{remark}
\label{rem22}
The results show, that $\V^{a\,K*}$ is not constructed by an independent
thinning of $V^{K*}$ where each $1$ of $V^{K*}$ would be
kept in $\V^{a\,K*}$ with probability $\alpha\in (0,1)$ given by (\ref{intl3}) (or 
deleted with a probability $1-\alpha$) independent of all 
the other $1$'s in $V^*$. 
For thinning in renewal processes 
see \cite{gs} 10.5.16 in Chapter 10. 
\end{remark}

\subsection{The zero cell of the renormalized STIT process is Bernoulli}
\label{sec02}

Let $Y=(Y_t: t>0)$ the be STIT process. We have that $\PP-$a.e.
for all $t>0$ the zero cell $\C_t$ contains the origin in its
interior. The set $\C_t$ is a random polytope. The mapping
$\T\to \K', Y_t\to \C_t$,
is a measurable mapping with respect to the Borel $\sigma-$fields
in $\T$ and $\K'$. The proof of the measurability is completely similar
to the one made for proving Theorem 10.3.2. in \cite{sw}. Moreover
the process $\C=(\C_t: t>0)$ inherits from $Y$ the property of having
c\`adl\`ag trajectories, i.e. $\C$ takes values on $D_{\K'}(\RR_+)$
and the mapping $D_{\T}(\RR_+)\to D_{\K'}(\RR_+)$, $\;
(Y_t:t\in \RR)\mapsto (\C_t:t\in \RR)$
is measurable.

\medskip

The process $\Gamma=(\Gamma_s:=a^s \C_{a^s}: s\in \RR)$
is the zero cell process of $\Z$, it
takes values in $D_{\K'}(\RR)$ and its law
$\mu^\Gamma$ of $\Gamma$ on $D_{\T}(\RR)$ is stationary.
The mapping
$$
\Theta: D_{\T}(\RR)\to D_{\K'}(\RR),\; (Z_t:t\in \RR)\mapsto (\Gamma_t:t\in
\RR)
$$
is a factor map, that is $\mu^\Gamma=\mu^\Z\circ \Theta^{-1}$.
Let $\Gamma^d=(\Gamma_n: n\in \ZZ)$ be the zero cell sequence
of $\Z^d$, its law on ${\K'}^\ZZ$ is noted $\mu^{\Gamma^d}$.
The mapping
$$
\Theta^d: \T^\ZZ\to {\K'}^\ZZ, \; (\Z_n:t\in \RR)\mapsto
(\Gamma_n: n\in \ZZ)
$$
is also a factor map.

\medskip

\begin{proposition}
\label{propber}
$D_{\K'}(\RR), (\Gamma_t:t\in \RR), \mu^\Gamma)$ is 
a Bernoulli flow of infinite entropy.
\end{proposition}

In \cite{mn1,m} it was shown that $\Z$ is a Bernoulli flow,
so, since $\Theta$ is a factor map sending $\Z$ on $\Gamma$
then the Bernoulli property of the zero cell process $\Gamma$
follows, see \cite{orns1} and \cite{orns2}. 

\medskip

The fact that the flow is of infinite entropy is a corollary
of the following result.

\begin{lemma}
\label{non-atomic}
The random variable $\Gamma_1$ is non-atomic. 
\end{lemma}

\begin{proof}
Let ${\cal F}$ denote the set of all closed subsets of $\RR^\ell$ and for $A\subset \RR^\ell$
${\cal F}^A=\{ F\in {\cal F}:A\cap F=\emptyset \}$ and ${\cal F}_A=\{ F\in {\cal F}:A\cap F\not=\emptyset \}$. Further, $B(x, n^{-1})$ denotes the ball with center $x$ and radius $n^{-1}$.
If $P$ is a polytope and $F_0 (P)$  the set of its vertices then
$$
\{ P\} = \K' \cap {\cal F}^{P^c}\cap \bigcap_{n\in \NN} \ 
\bigcap_{x\in F_0(P)} {\cal F}_{B(x, n^{-1})}
$$
which implies that the singleton $\{ P\} \in {\cal B}(\K')$ 
(see \cite{sw}, Section 2.1).

\medskip  

Hence, if we assume that $\Gamma_1$ is atomic then, because 
$\Gamma_1 \sim \Gamma_0 =\C_1$, there exists a polytope $P$ such that
$\PP(\C_1 =P)>0$. Regarding the construction of STIT tessellations, 
this implies for all $(\ell -1)$-dimensional faces $f$ of $P$ 
and their 'carrying' hyperplanes $h(f)\in \hH$ with $h(f)\supset f$ 
that $\Lambda (\{ h(f)\})>0$. But this contradicts the property 
$\Lambda (\{ h\})=0$ for any hyperplane $h\in \hH$ which is a consequence 
of the translation invariance of the measure  $\Lambda$.
\end{proof}

\medskip

For completeness we will give a brief sketch of the proof that
$\Z$ is Bernoulli, and as we will see it shares many ideas with 
the regeneration properties as it is pointed out in Remark 
\ref{rem1x}. 

\medskip

Let $Y$ be an STIT process. We can assume that a.e. for all $t>0$, 
the origin $0$ belongs to the interior of the zero cell
$\C_t^1$ in $Y_t$.

\medskip

Denote by $\xi=\xi^1$ and $\xi_W=\xi^1_W$
the laws of $\Z_0$ and $\Z_0\wedge W$
respectively. Define the product probability measures,
$\varrho=\xi^{\otimes \NN}$ and $\varrho_W=\xi_W^{\otimes \NN}$.
Take a random sequence
${\cal R}=({\vec R}_n: n\in \ZZ)$ independent of $Y$, and distributed
as ${\cal R}\sim \varrho^\ZZ$. So, the components
$({\vec R}_n:=(R^m_n: m\in \NN): n\in \NN)$ are independent with
${\vec R}_n\sim \varrho$ for all $n$, and so the components
$(R^m_n: m\in \NN)$ are independent with $R^m_n\sim \xi$ for all $m$.

\medskip

The shift transformations $\sigma$ and $\sigma^{-1}$ act on the
sequences ${\cal R}$ by $\sigma({\cal R})=
({\vec R}_{n+1}: n\in \ZZ)$ and $\sigma^{-1}({\cal R})=
({\vec R}_{n-1}: n\in \ZZ)$, and they preserve $\varrho^\ZZ$.
We have ${\cal R}\wedge W\sim \varrho_W^\ZZ$, where 
${\cal R}\wedge W:=({\vec R}_n\wedge W: n\in \ZZ)$. By using this
representation we get $({\cal R}\wedge W)\wedge W'\sim 
\varrho_{W'}^\ZZ$
for every pair of windows $W$, $W'$ such that $W'\subseteq W$.
  
\medskip

Let $W$ be a window containing the origin in its interior.
We have $\PP(\partial R^1_{j}\cap 
\Int(W)=\emptyset)=e^{-\Lambda([K])}$.   
Define the set
$$
{\cal E}^K=\{{\cal R}: \, \partial R^1_j\cap
\Int((a-1) a^{-(j+1)}K)=\emptyset, \;\forall j\le 0\}. 
$$
The event $\{{\cal R}\in {\cal E}^K\}$ only depends
on ${\cal R}^-=({\vec R}_n: n\le 0)$. We have
\begin{equation}
\label{medek}
\PP({\cal E}^K)
=\prod_{j\ge 0}
e^{-(a^{-j}-a^{-(j+1)})\Lambda([K])}=e^{\Lambda([K])}>0.
\end{equation}
Let $\tau^K=(\tau^K_i: i\in \ZZ)$ be the ordered sequence
of random times for which $\sigma^{\tau^K_i}({\vec R})\in {\cal E}^K$
and where $\tau^K_0=\inf\{\tau^K_i: \tau^K_i\ge 0\}$.
From the Birkhoff Ergodic Theorem and since $\PP({\cal E}^K)>0$, 
this sequence takes finite values a.s..

\begin{remark}
\label{rem1x}
Note that ${\vec R}_n$ has the same distribution as 
${\vec Y}^{'(n)}$. The relations (\ref{eqcap} and (\ref{medek})
suggest that we can couple the events ${\cal E}^K$ and
$\{\Z_0\supset K\}$, and so the sequence $\tau^K$ 
will correspond to $\V^K$. 
\end{remark}

\medskip

Let $i>0$. Let us fix $Q\in \T$ such that $Q\wedge W=\{W\}$.
We define $\varphi_{W}^i(\tau^K_{-i})=\{W\}$ and
\begin{equation}
\label{itervar} 
\forall k\ge 0: \;\,
(\!\varphi_{W}^i({\cal R}\!\wedge \!W)\!)_{\tau^K_{-i}+k} \!=\!
(a^k Q \boxplus_{i=1}^{k}
\frac{a^{k+1-i}}{a-1}{\vec R}_i)\wedge \! W,
\end{equation}
From the definition of $\tau^K$ it is straightforward to check that
for all $j\le i$ we have $(\varphi_W^i({\cal R}\wedge
W))_{\tau^K_{-j}}=\{W\}$.
Moreover, for $j\ge i$, the cells in $(\varphi_{W}^j({\cal R}\wedge W))_n$
can be enumerated in the same way
as those in $(\varphi_{W}^i({\cal R}\wedge W))_n$
for $n\ge \tau^K_{-i}$. Then,
$$
\forall \, j\ge i\,,\; \forall n\ge \tau^K_{-i} :\;\;\;
(\varphi_{W}^j({\cal R}\wedge W))_n=(\varphi_{W}^{i}({\cal R}\wedge
W))_n\,.
$$
Hence, a sequence
$\varphi_W({\cal R}\wedge W)=(\varphi_W({\cal R}\wedge W)_n:n\in \ZZ)$
is $\; \varrho^{\otimes \ZZ}$-a.e. well-defined
by the following equality:
\begin{equation}
\label{fico22}
 \forall n\ge \tau^K_{-i}
:\;\;\; (\varphi_{W}({\cal R}\wedge W))_n=
(\varphi_{W}^i({\cal R}\wedge W))_n\,.
\end{equation}   
 
Let us take the window $W'$ with $W\subset \Int W'$. We have
$\{\tau^{W'}_i: i\in \ZZ\} \subseteq \{\tau^K_i:i \in \ZZ\}$ 
and so $\tau^{W'}_{-i}\le \tau^K_{-i}$ for all $i>0$.
It can be shown that the enumerations of the cells in
$(\varphi_{W'}({\cal R}\wedge W'))_n\wedge W$
and $(\varphi_{W}({\cal R}\wedge W))_n$ can
use the same order. From (\ref{itervar}) we get,
$(\varphi_W)_n=(\varphi_{W'})_n\wedge W$
$\; \varrho^{\otimes \ZZ}$ a.e.. Therefore,
\begin{equation}
\label{cubre1}
\varphi_{W'}\wedge W=\varphi_W\;\;
\varrho^{\ZZ}-\hbox{a.e.}\,.
\end{equation}
This construction can be made for a sequence of
increasing windows $(W_k: k\in \NN)$ with $W_k\subseteq \Int
W_{k+1}$ and $W_k\nearrow \RR^\ell$. 
From (\ref{cubre1}) and Theorem $2.3.1.$ in \cite{sw},
there exists a function $\varphi$ taking values in $\T^\ZZ$,
defined $\varrho^{\ZZ}$-a.e. and such that for all $k\ge 1$,
$\varphi\wedge W_k=\varphi_{W_k}$
$\;\, \varrho^{\ZZ}-$a.e..
It can be shown, see \cite{mn1, m} for details, that it is satisfied
$$
\sigma_{\T}\circ \varphi=
\varphi\circ \sigma_{\T} \varrho^{\ZZ}-\hbox{a.e.}  \hbox{ and }
\varrho^{\otimes \ZZ}\circ \varphi^{-1}=\mu^{\Z^d}.
$$
Then, $\varphi$ is a factor which is non-anticipating because
$(\varphi({\cal R}))_n$ only depends on
$({\cal R}_j: j\le n)$.
Then $(\T^\ZZ, \mu^{\Z^d}, \sigma_\T)$
is a factor of Bernoulli shift and from Ornstein theory we get
that it  is also Bernoulli, see \cite{orns1} and \cite{orns2}.
It has infinite entropy, see \cite{mn1}.
By using Theorem $4$ in Section $12$, part $2$ in \cite{orns2} and
also \cite{orns1}, we get that the time continuous
process $\Z$ is a Bernoulli flow.

\medskip

\begin{corollary}
\label{cor1}
For all compact convex sets $K\subset \RR^\ell$ we have,
$$
\lim\limits_{n\to \infty} \frac{1}{n}\sum_{l=0}^{n-1}
\1_{\Gamma_n\supset a^{-n}K}=e^{-\Lambda[K]}
=
\lim\limits_{n\to \infty} \frac{1}{n}\sum_{l=0}^{n-1}
\1_{\C_n\supset a^{-n}K}.
$$
\end{corollary}

\begin{proof}
Since $\Z^d$ is ergodic we can apply the Birkhoff Ergodic Theorem,
and so, the result follows from the equalities
$$
\{\C_n\supset a^{-n}K\}=\{\Gamma_n\supset a^{-n}K\}=
\{\partial Z_n\cap \Int K=\emptyset\}.
$$
\end{proof}

\section*{Acknowledgments} The authors thank  
the support of Program Basal CMM from CONICYT (Chile) and the DAAD
(Germany).

\end{document}